\documentclass[reqno, 12pt]{article}

\usepackage{enumerate}
\usepackage{needspace}
\usepackage{latexsym}
\usepackage{amsmath}
\usepackage{amsfonts}
\usepackage{amsthm}
\usepackage{amssymb,mathtools}
\usepackage{newlfont}
\usepackage{graphics}
\usepackage{color}
\usepackage{float}
\usepackage{diagbox}
\usepackage{tocloft}
\usepackage{titlesec}
\usepackage{booktabs,array,longtable,enumitem,microtype}
\usepackage{extpfeil}
\usepackage{centernot}

\usepackage[linesnumbered,ruled,vlined]{algorithm2e}
\usepackage{url}

\usepackage[margin=1in]{geometry}

\usepackage[T1]{fontenc}
\usepackage{lmodern}
\usepackage{microtype}
\usepackage{rotating}
\usepackage{multirow}
\usepackage{extarrows}
\usepackage[sort,compress,numbers]{natbib}
\usepackage[utf8]{inputenc}
\usepackage{xcolor}
\usepackage{listings}
\usepackage{aliascnt}
\usepackage[pagebackref,colorlinks=true,linkcolor=blue,citecolor=red,urlcolor=blue]{hyperref}
\usepackage[nameinlink,noabbrev,capitalize]{cleveref}
\numberwithin{equation}{section}

\setlist[enumerate]{label=\textup{(\roman*)},leftmargin=*}
\newtheorem{theorem}{Theorem}[section]
\newaliascnt{proposition}{theorem}
\newtheorem{proposition}[proposition]{Proposition}
\aliascntresetthe{proposition}
\newaliascnt{lemma}{theorem}
\newtheorem{lemma}[lemma]{Lemma}
\aliascntresetthe{lemma}
\newaliascnt{corollary}{theorem}
\newtheorem{corollary}[corollary]{Corollary}
\aliascntresetthe{corollary}

\theoremstyle{definition}
\newaliascnt{definition}{theorem}

\aliascntresetthe{definition}
\newaliascnt{remark}{theorem}

\aliascntresetthe{remark}
\newaliascnt{example}{theorem}

\aliascntresetthe{example}
\newaliascnt{conjecture}{theorem}

\aliascntresetthe{conjecture}
\newaliascnt{question}{theorem}

\aliascntresetthe{question}
\newaliascnt{problem}{theorem}
\newtheorem{problem}[problem]{Problem}
\aliascntresetthe{problem}
\newaliascnt{alg}{theorem}

\aliascntresetthe{alg}

\allowdisplaybreaks[4]

\SetKwInput{KwInput}{Input}                
\SetKwInput{KwOutput}{Output}              

\DeclareMathOperator{\Vol}{Vol}
\DeclareMathOperator{\htop}{ht}

\DeclareMathOperator{\relint}{relint}

\DeclareMathOperator{\Ehr}{Ehr}

\newcommand{\hstar}{h^{*}}
\newcommand{\floor}[1]{\left\lfloor #1\right\rfloor}

\title{Unimodality for IDP Lattice Simplices of Prime Normalized Volume}

\author{Feihu Liu$^{\color{blue} \dag}$ and Zihao Zhang$^{\color{blue} \S}$
\\[2mm]
{\small $^{\color{blue} \dag}$ Center for Combinatorics, LPMC,}\\[-0.8ex]
{\small Nankai University, Tianjin 300071, P.R.~China}\\
{\small $^{\color{blue} \S}$ School of Mathematics and Statistics,}\\[-0.8ex]
{\small Beijing Institute of Technology, Beijing 102400, P.R.~China}\\
{\small {\color{blue} $^\dag$} Email address: liufeihu7476@163.com}\\
{\small {\color{blue} $^\S$} Email address: zihao-zhang@foxmail.com}\\
}

\date{\today}

\begin{document}

\maketitle

\begin{abstract}
Recently, Ferroni constructed a family of counterexamples to the well-known conjecture in Ehrhart theory stating that the $h^*$-polynomial of a lattice polytope with the integer decomposition property is unimodal. This raises the question of whether the $h^*$-polynomial of a lattice simplex with the integer decomposition property remains unimodal. In this note, we prove that every lattice simplex with the integer decomposition property and prime normalized volume has a unimodal $h^*$-polynomial. Furthermore, we establish several sufficient conditions for the unimodality of the $h^*$-polynomial of such simplices.
\end{abstract}

\noindent
\begin{small}
\end{small}

\noindent
\begin{small}
\end{small}


\section{Introduction}

Let $M$ be a lattice, that is, a free abelian group of finite rank.
Let $M_{\mathbb{R}} \coloneqq M \otimes_{\mathbb{Z}} \mathbb{R}$. 
Here $M \otimes_{\mathbb{Z}} \mathbb{R}$ denotes the tensor product over the integer ring $\mathbb{Z}$ of the $M$ (that is, a $\mathbb{Z}$-module) and the real field $\mathbb{R}$ (also considered as a $\mathbb{Z}$-module). 
This algebraic operation is known as extension of scalars.

A \emph{lattice polytope} $P \subset M_{\mathbb{R}}$ is defined as the convex hull of a finite set of points in $M$. 
For a $d$-dimensional lattice polytope $P$, its \emph{lattice-point enumerator}
$$L_P(n)=|nP\cap M| \quad (n\geq 0)$$
counts the number of integer points in the $n$-th dilation $nP = \{ n\alpha : \alpha \in P \}$ of $P$.
Ehrhart~\cite{Ehr62} proved that the function $L_P(n)$ is a polynomial in $n$ of degree $d$ with constant term $1$.
This polynomial is called the \emph{Ehrhart polynomial} of $P$.

The \emph{Ehrhart series} of $P$ has the form
\[\Ehr_P(t)=\sum_{n\geq0}|nP\cap M|t^n=\frac{\hstar_P(t)}{(1-t)^{d+1}},\qquad \hstar_P(t)=\sum_{i=0}^d h_i^*(P)t^i.
\]
The numerator $h_P^*(z)$ is the \emph{$h^*$-polynomial} of $P$, and its coefficient vector is the \emph{$h^*$-vector}. 
The coefficients of $h^*$-polynomial are nonnegative integers \cite{Sta80}. 

For further background on lattice polytopes, we refer to several excellent books~\cite{BR15,Zie95} and the good surveys \cite{Bra16,FH24,Liu19}.

A finite sequence $(a_0,\ldots,a_d)$ is \emph{unimodal} if some $k$ satisfies $a_0\leq\cdots\leq a_k\geq\cdots\geq a_d$.
We say that $P$ possesses the \emph{integer decomposition property} (IDP) if, for any positive integer $n$, every lattice point $z \in n\mathcal{P} \cap M$ can be expressed as a sum $z = z_1 + \dots + z_n$, where each $z_i \in \mathcal{P} \cap M$. The polytope $\mathcal{P}$ is called \emph{very ample} if it satisfies the integer decomposition property for all sufficiently large $n$.

Our initial motivation for this paper comes from the following open problem.

\begin{problem}\label{OpenIDPUniom}
Let $P$ be a lattice polytope possessing the integer decomposition property. Is it true that $h^*$-polynomial is unimodal?
\end{problem}

This problem was proposed by Schepers and Van Langenhoven \cite{SVL13}. Recently, Ferroni \cite{Fer26} gave a family of counterexamples to this open problem. We list the related progress and results in chronological order as follows:
\begin{enumerate}
    \item In 1989, Stanley \cite{Sta89} posed a problem more general than \cref{OpenIDPUniom}: is it true that every standard graded Cohen-Macaulay domain has a unimodal (or log-concave) $h$-vector?
        
    \item In 1989, Hibi \cite{Hibi89} again recorded Stanley's problem and obtained partial results.
    
    \item In 1992, Hibi \cite{Hibi92} had conjectures that all reflexive polytopes had unimodal $h^*$-polynomials.
    
    \item In 2005, Musta\c{t}\v{a} and Payne \cite{MP05} gave a counterexample to Hibi's conjecture (above case~3 \cite{Hibi92}).
     In 2008, Payne \cite{Pay08} proved that Hibi's conjecture has a counterexample in every dimension $d\geq 6$.
    
    \item In 2006, Ohsugi and Hibi \cite{OH06} posed the following conjecture in the context of normal polytopes: If $P$ is Gorenstein and IDP, then $h^*$-polynomial of $P$ is unimodal.
    
    \item In 2013, Schepers and Van Langenhoven \cite{SVL13} explicitly posed \cref{OpenIDPUniom} in the context of lattice polytopes. They also obtained partial results; for example, \cref{OpenIDPUniom} holds for $d\leq 4$.
        
     \item In 2016, Braun's survey \cite{Bra16} again records \cref{OpenIDPUniom} and related developments.

     \item In 2022, Adiprasito, Papadakis, Petrotou, and Steinmeyer \cite{APPS22} proved that Gorenstein IDP polytopes are indeed $h^*$-unimodal.

     \item In 2024, Ferroni and Higashitani's \cite{FH24} survey once again records progress on \cref{OpenIDPUniom} and related problems.

     \item In 2025, Hofscheier, Kurylenko, and Nill \cite{HKuN25} constructed examples of IDP polytopes whose $h^*$-polynomials are not log-concave (log-concavity is stronger than unimodality).
         
     \item In 2026, Hofscheier, Kurylenko, and Nill \cite{HKuN26} provided an example of a very ample lattice polytope with a non-unimodal $h^*$-polynomial.
         
     \item In 2026, Ferroni \cite{Fer26} gave a family of counterexamples to \cref{OpenIDPUniom}.
\end{enumerate}

In fact, the motivation for our study is to ask whether \cref{OpenIDPUniom} holds for \textbf{lattice simplices}.
A $d$-dimensional lattice polytope with exactly $d+1$ vertices is called a $d$-dimensional \emph{lattice simplex}.
We define the \emph{normalized volume} $\operatorname{Vol}(P)$ of $P$ to be $d!$ times the standard volume, normalized such that a fundamental parallelepiped of the lattice $M_P$ has volume one.

The main result of this paper is the following theorem.

\begin{theorem}\label{thm:main}
Let $\Delta$ be a lattice simplex with the integer decomposition property. If its normalized volume is a prime number, then its
$h^*$-vector is unimodal.
\end{theorem}

The paper is organized as follows. 
\cref{Section-2}, we state some preliminaries; in particular, the inequality result of Adiprasito--Papadakis--Petrotou\cite{APP25} will be used in the subsequent proofs.
In \cref{sec:box}, we prove some lemmas needed for the proofs of the main theorems.
In \cref{Section--4}, we prove \cref{thm:main} by using the result on the shifted symmetry (i.e., $h_i^*=h_{d+1-i}^*$ for $1\leq i\leq d$) of the coefficients of Higashitani \cite{Hig10}.
In \cref{Section--5}, we provide some sufficient conditions for \cref{OpenIDPUniom} to hold on lattice simplices.

\section{Preliminaries}\label{Section-2}

\subsection{Homogenization}

Suppose that $P$ has dimension $d$. For any chosen lattice point $v \in P \cap M$, we define the sublattice parallel to $P$ by
\[M_P \coloneqq M \cap \operatorname{span}_{\mathbb{R}}(P - v).
\]
The affine lattice restricted to the affine span $\operatorname{aff}(P)$ is given by $v + M_P$. Up to a translation by $-v$--(an operation that naturally preserves both the Ehrhart counting function and the integer decomposition property (IDP) of $P$)--we may assume without loss of generality that $P$ is full-dimensional in the vector space $M_P \otimes_{\mathbb{Z}} \mathbb{R}$. 
Throughout this paper, geometric notions such as the interior, boundary, and facets of a polytope are taken relative to its affine span. In particular, $\partial P$ denotes the relative boundary of $P$.

We define the \emph{normalized volume} $\operatorname{Vol}(P)$ to be $d!$ times the standard volume, normalized such that a fundamental parallelepiped of the lattice $M_P$ has volume one. Under this normalization, a simplex is said to be \emph{unimodular} if its normalized volume is exactly one; by convention, a single lattice point (regarded as a zero-dimensional simplex) has normalized volume one.

To algebraically encode the lattice points in dilations of $P$, we construct the cone over $P$ as follows:
\[C_P \coloneqq \mathbb{R}_{\geq 0} (P \times \{1\}).
\]
Associated with this cone are the ambient lattice $N_P \coloneqq \operatorname{span}_{\mathbb{R}}(C_P) \cap (M \oplus \mathbb{Z})$ and the affine semigroup $S_P \coloneqq C_P \cap N_P$. The final coordinate of an element in $S_P$ is referred to as its \emph{height}, which we denote by $\operatorname{ht}$. By definition, the elements of $S_P$ at height $n \in \mathbb{Z}_{\geq 0}$ are precisely the vectors of the form $(x, n)$, where $x \in nP \cap M$; notably, the only element at height zero is the origin. In this framework, the geometric condition that $P$ possesses the IDP is algebraically equivalent to the assertion that the additive semigroup $S_P$ is generated by its elements of height one.

\subsection{Elementary identities and inequalities}

For a lattice polytope $P$, let $\relint(P)$ denote its relative interior. 
The reciprocity theorem of Ehrhart and Macdonald is a classical result in polytope theory. 
For every positive integer $n$, we have 
\begin{align*}
L_P(-n)=(-1)^d\bigl|\relint(nP)\cap M\bigr|\qquad(n\geq1).
\end{align*} 
In particular, if $d$ is even, then $L_P(-1)\geq0$.
Its full general form is due to Macdonald \cite{Macdonald1971}; see also \cite[Theorem~4.1]{BR15}.
The leading coefficient of $L_P(n)$ is $\Vol(P)/d!$.

Expanding $(1-t)^{-d-1}$ and comparing coefficients gives
\begin{equation}\label{eq:binomial}
 L_P(n)=\sum_{i=0}^d h_i^*(P)\binom{n+d-i}{d}.
\end{equation}
For nonnegative $n$, terms with $i>n$ vanish. Evaluating at $n=0,1$
and comparing leading coefficients, respectively, gives
\begin{equation}\label{eq:basic}
 h_0^*(P)=1,\qquad
 h_1^*(P)=|P\cap M|-d-1,\qquad
 \sum_{i=0}^d h_i^*(P)=\Vol(P).
\end{equation}
For $d\geq1$, evaluate the polynomial identity \eqref{eq:binomial}
at $n=-1$. All terms with $i<d$ vanish and
$\binom{-1}{d}=(-1)^d$. Reciprocity therefore gives
\begin{equation}\label{eq:last}
 h_d^*(P)=|\relint(P)\cap M|.
\end{equation}
The \emph{Ehrhart degree} of $P$ is $s(P)=\deg\hstar_P(t)$.
We set $h_i^*(P)=0$ for $i>d$ and retain terminal zeros when discussing
the full $h^*$-vector.

\begin{theorem}{\em (Adiprasito--Papadakis--Petrotou; \cite[Corollary~2.2]{APP25})}\label{ext:tail}
If $P$ is an IDP lattice polytope of dimension $d$, then
\begin{equation}\label{eq:tail}
 h_{\floor{(d+1)/2}}^*(P)\geq\cdots\geq h_d^*(P)\geq0.
\end{equation}
\end{theorem}

\begin{theorem}{\em (Stanley; \cite[Proposition~3.4]{Sta91})}\label{ext:stanley}
Let $P$ be an IDP lattice polytope. Let $s=s(P)$ be the Ehrhart degree of $P$.
For integers $i,j\geq1$ with $i+j<s$, we have
\begin{equation*}
h_1^*(P)+h_2^*(P)+\cdots +h_i^*(P) \leq\ h_{j+1}^*(P)+h_{j+2}^*(P)+\cdots +h_{j+i}^*(P).
\end{equation*}
In particular, if $s\geq3$, then $h_1^*(P)\leq h_2^*(P)$.
\end{theorem}

This result of Stanley (\cref{ext:stanley}) is also recorded in \cite[Theorem~1.2]{HKN18}.
This paper extends the result to spanning lattice polytopes in \cite[Theorem~1.4]{HKN18}. We only need the IDP case.

\section{Some lemmas}\label{sec:box}

Let $\Delta = \operatorname{conv}(v_0, \ldots, v_d)$ be a $d$-dimensional lattice simplex. For each $j \in \{0, \dots, d\}$, we lift the vertex to height one by setting $w_j \coloneqq (v_j, 1)$, and we define the associated sublattice as $L \coloneqq \bigoplus_{j=0}^d \mathbb{Z} w_j$. Because the vertices $v_0, \ldots, v_d$ are affinely independent, their lifted counterparts $w_0, \ldots, w_d$ are linearly independent. Consequently, the linear map
\[W \colon \mathbb{R}^{d+1} \longrightarrow \operatorname{span}_{\mathbb{R}}(C_\Delta), \quad \lambda \longmapsto \sum_{j=0}^d \lambda_j w_j
\]
is a vector space isomorphism.

Using this isomorphism, we define the coordinate set corresponding to the fundamental parallelepiped of the cone $C_\Delta$ as
\begin{equation}\label{eq:box}
    \Lambda_\Delta \coloneqq \left\{ \lambda \in [0,1)^{d+1} \;\middle|\; W(\lambda) \in N_\Delta \right\}.
\end{equation}
The set $\Lambda_\Delta$ is naturally endowed with an abelian group structure via coordinatewise addition modulo $1$, which canonically identifies $\Lambda_\Delta$ with the quotient group $N_\Delta / L$. For any element $\lambda = (\lambda_0, \dots, \lambda_d) \in \Lambda_\Delta$, we define its \emph{height} and \emph{support} by
\[
    \operatorname{ht}(\lambda) \coloneqq \sum_{j=0}^d \lambda_j \quad \text{and} \quad \operatorname{supp}(\lambda) \coloneqq \{j \mid \lambda_j \neq 0\},
\]
respectively.

\begin{lemma}{\em (Hibi; \cite[Proposition 27.7]{Hibi92})}\label{lem:box}
The set $\Lambda_\Delta$ is a finite abelian group, and
\begin{equation}\label{eq:box-enumeration}
|\Lambda_\Delta|=\Vol(\Delta),\qquad \hstar_\Delta(t)=\sum_{\lambda\in\Lambda_\Delta}t^{\htop(\lambda)}.
\end{equation}
Its heights are integers between $0$ and $d$, and only the zero element has height zero.
\end{lemma}

\cref{lem:box} is also recorded in \cite[Lemma~2.1]{Hig14}.

\begin{lemma}\label{lem:nocarry}
The simplex $\Delta$ possesses the integer decomposition property (IDP) if and only if every nonzero element $\lambda \in \Lambda_\Delta$ of height $n \coloneqq \operatorname{ht}(\lambda)$ admits a decomposition
\begin{equation}\label{eq:nocarry}
    \lambda = \mu^{(1)} + \cdots + \mu^{(n)} \quad \text{in } \mathbb{R}^{d+1},
\end{equation}
where $\mu^{(i)} \in \Lambda_\Delta$ and $\operatorname{ht}(\mu^{(i)}) = 1$ for all $1 \leq i \leq n$.
\end{lemma}
\begin{proof}
Let $\lambda \in \Lambda_\Delta$ be a nonzero element with $\operatorname{ht}(\lambda) = n$. By definition, $W(\lambda) = (x, n)$ for some $x \in n\Delta \cap M$. Because $\Delta$ has the IDP, the cone element $(x, n)$ can be written as a sum of $n$ height-one lattice points in the cone; that is, $(x, n) = \sum_{i=1}^n (x_i, 1)$, where $x_i \in \Delta \cap M$. 
For each $1 \leq i \leq n$, let $\mu^{(i)}$ denote the barycentric coordinates of $x_i$ with respect to the vertices of $\Delta$. 

By construction, the components of $\mu^{(i)}$ are non-negative, sum to one (hence $\operatorname{ht}(\mu^{(i)}) = 1$), and satisfy $W(\mu^{(i)}) = (x_i, 1)$. Because the columns of the linear isomorphism $W$ are linearly independent, the relation $W(\lambda) = \sum_{i=1}^n W(\mu^{(i)})$ implies the exact equality $\lambda = \sum_{i=1}^n \mu^{(i)}$ in $\mathbb{R}^{d+1}$. Furthermore, since $\lambda \in \Lambda_\Delta$, we have $\lambda_j < 1$ for all coordinates $j$. The non-negativity of the coordinates $\mu_j^{(i)}$ thus forces $0 \leq \mu_j^{(i)} \leq \lambda_j < 1$. Consequently, $\mu^{(i)} \in [0,1)^{d+1}$, which implies $\mu^{(i)} \in \Lambda_\Delta$, thereby establishing \eqref{eq:nocarry}.

Conversely, suppose the decomposition condition holds. Let $z \in S_\Delta$ be an arbitrary element of the affine semigroup. By standard fundamental parallelepiped decomposition, $z$ can be expressed uniquely as
\[z = W(\lambda) + \sum_{j=0}^d a_j w_j,
\]
where $\lambda \in \Lambda_\Delta$ and $a_j \in \mathbb{Z}_{\geq 0}$. Applying the linear height function yields $\operatorname{ht}(z) = \operatorname{ht}(\lambda) + \sum_{j=0}^d a_j$. If $\lambda \neq 0$, we invoke \eqref{eq:nocarry} to decompose $\lambda = \sum_{i=1}^{\operatorname{ht}(\lambda)} \mu^{(i)}$, where each $\mu^{(i)} \in \Lambda_\Delta$ has height one. It follows that $W(\lambda) = \sum_{i=1}^{\operatorname{ht}(\lambda)} W(\mu^{(i)})$. Since $\mu^{(i)} \in \Lambda_\Delta$ and its mapped image is an integer vector, each $W(\mu^{(i)})$ constitutes a height-one element of $S_\Delta$. The generators $w_j$ are naturally height-one elements of $S_\Delta$ as well. Thus, $z$ is represented entirely as a sum of $\operatorname{ht}(\lambda) + \sum_{j=0}^d a_j = \operatorname{ht}(z)$ height-one elements. This confirms that $S_\Delta$ is additively generated by its height-one elements, proving that $\Delta$ has the IDP.
\end{proof}

Generation of $\Lambda_\Delta$ by its height-one elements as a group is not the assertion of \cref{lem:nocarry}. The equality in \eqref{eq:nocarry} must hold before reduction modulo integers, and the number of summands is prescribed by height.

\begin{lemma}\label{lem:initial}
Let $P$ be a lattice polytope possessing the integer decomposition property (IDP). Then either $h^*_P(t) = 1$, or $1 = h^*_0(P) \leq h^*_1(P)$. Furthermore, if an IDP simplex contains no lattice points other than its vertices, it is unimodular.
\end{lemma}

\begin{proof}
Let $\Delta$ be an IDP simplex whose only lattice points are its vertices. By formula \eqref{eq:basic}, this implies $h^*_1(\Delta) = 0$. Suppose, for the sake of contradiction, that the fundamental parallelepiped contained a non-zero element. Because $\Delta$ has the IDP, \cref{lem:nocarry} would guarantee the existence of an element in $\Lambda_\Delta$ of height exactly one. However, this contradicts the fact that the number of height-one elements in $\Lambda_\Delta$ is precisely enumerated by $h^*_1(\Delta) = 0$. Consequently, the coordinate set $\Lambda_\Delta$ consists solely of the origin. It follows that the associated quotient group is trivial. Since the order of this group is equal to the normalized volume of $\Delta$, we conclude that the normalized volume is one, proving that $\Delta$ is unimodular.

Now consider an arbitrary $d$-dimensional lattice polytope $P$, and suppose $h^*_1(P) = 0$. Formula \eqref{eq:basic} dictates that $P$ contains exactly $d+1$ lattice points in total. Since every vertex of a lattice polytope is a lattice point, and any $d$-dimensional polytope must have at least $d+1$ vertices, $P$ is forced to be a simplex whose vertices constitute its entire set of lattice points. By \cref{lem:box}, we deduce that $P$ is a unimodular simplex, which yields $h^*_P(t) = 1$. Otherwise, if $h^*_1(P) \neq 0$, it must be a positive integer. Since $h^*_0(P) = 1$ by definition, this immediately establishes the inequality $1 = h^*_0(P) \leq h^*_1(P)$.
\end{proof}

\begin{lemma}\label{lem:faces}
Let $P$ be a lattice polytope possessing the integer decomposition property (IDP). Then every face of $P$ is also IDP with respect to its induced lattice. Furthermore, let $\Delta = \operatorname{conv}(v_0, \ldots, v_d)$ be a simplex, and for any nonempty subset $I \subseteq \{0, \ldots, d\}$, define the face $\Delta_I \coloneqq \operatorname{conv}(v_i \mid i \in I)$. The natural embedding obtained by appending zero coordinates induces a height-preserving bijection
\begin{equation}\label{eq:face-group}
    \Lambda_{\Delta_I} \cong \left\{ \lambda \in \Lambda_\Delta \;\middle|\; \lambda_j = 0 \text{ for } j \notin I \right\}.
\end{equation}
\end{lemma}

\begin{proof}
Let $F$ be a proper face of $P$. By definition, there exists a linear functional $u$ and a constant $c \in \mathbb{R}$ such that $F = \{ x \in P \mid u(x) = c \}$, and the valid inequality $u(y) \leq c$ holds for all $y \in P$. Suppose $x \in nF \cap M$ for some positive integer $n$. Because $P$ has the IDP and $x \in nP \cap M$, there exists a decomposition $x = x_1 + \cdots + x_n$ where $x_k \in P \cap M$ for each $1 \leq k \leq n$. Evaluating the linear functional $u$ at $x$ yields
\[n c = u(x) = \sum_{k=1}^n u(x_k).
\]
Since $x_k \in P$, we have the upper bound $u(x_k) \leq c$ for all $k$. In order for the sum to equal $nc$, this inequality must be sharp; that is, $u(x_k) = c$ for every $k$. This forces every summand $x_k$ to lie in $F \cap M$. The conclusion that $F$ is IDP with respect to its induced lattice then follows naturally upon translating $F$ by one of its lattice vertices.

For the simplex case, consider the ambient lattice associated with the cone over $\Delta_I$, given by $N_I \coloneqq \operatorname{span}_{\mathbb{R}} \{w_i \mid i \in I\} \cap (M \oplus \mathbb{Z})$. Let $\lambda$ be a vector supported on $I$, and let $\tilde{\lambda}$ be the corresponding vector in $\mathbb{R}^{d+1}$ obtained by inserting zeros at all coordinates $j \notin I$. Under the linear isomorphism $W$, the image $W(\tilde{\lambda})$ belongs to $N_\Delta$ if and only if it belongs to $N_I$, as the linear combination is restricted to the generators $\{w_i \mid i \in I\}$. The operation of adding or removing these zero coordinates trivially leaves the height $\operatorname{ht}(\lambda)$ invariant, which establishes the canonical identification in \eqref{eq:face-group}.
\end{proof}

\begin{proposition}\label{prop:core}
Suppose that the normalized volume satisfies $\operatorname{Vol}(\Delta) > 1$. Define the index set $I$ by
\[I \coloneqq \bigcup_{\lambda \in \Lambda_\Delta} \operatorname{supp}(\lambda),
\]
and let $F \coloneqq \Delta_I$ be the corresponding face with dimension $r \coloneqq |I| - 1$. Then $r \geq 1$, and the following equalities hold:
\begin{equation}\label{eq:core}
    \Lambda_F \cong \Lambda_\Delta, \qquad h^*_F(t) = h^*_\Delta(t), \qquad \operatorname{Vol}(F) = \operatorname{Vol}(\Delta).
\end{equation}
Furthermore, no coordinate of $\Lambda_F$ is identically zero. If $\Delta$ possesses the integer decomposition property (IDP), then $F$ also possesses the IDP.
\end{proposition}

\begin{proof}
The assumption $\operatorname{Vol}(\Delta) > 1$ ensures that the coordinate set $\Lambda_\Delta$ contains a nonzero element. Such an element cannot have a support of size exactly one; if it did, its height would lie in the open interval $(0, 1)$, which contradicts the requirement that the height of any element mapped into the lattice $N_\Delta$ must be an integer. Therefore, the support of any nonzero element in $\Lambda_\Delta$ must contain at least two indices, yielding $|I| \geq 2$ and consequently $r \geq 1$.

By the definition of $I$, every element $\lambda \in \Lambda_\Delta$ satisfies $\lambda_j = 0$ for all $j \notin I$. As a result, the natural bijection established in \eqref{eq:face-group} canonicaly identifies the entire group $\Lambda_\Delta$ with $\Lambda_F$. By \cref{lem:box}, this structural isomorphism directly yields the remaining two equalities in \eqref{eq:core}. Moreover, the construction of $I$ guarantees that for each retained index $i \in I$, there exists at least one element in $\Lambda_\Delta$ whose $i$-th coordinate is nonzero; hence, no coordinate of $\Lambda_F$ is identically zero. Finally, if $\Delta$ is IDP, the fact that $F$ is a proper or improper face of $\Delta$, combined with \cref{lem:faces}, guarantees that $F$ is also IDP.
\end{proof}

\section{Shifted symmetry and the prime-volume theorem}\label{Section--4}

\begin{lemma}\label{lem:inversion}
For any element $\lambda \in \Lambda_\Delta$, let $-\lambda \in \Lambda_\Delta$ denote its additive inverse in the coordinate group. Then
\begin{equation}\label{eq:inverse}
\operatorname{ht}(\lambda) + \operatorname{ht}(-\lambda) = |\operatorname{supp}(\lambda)|.
\end{equation}
In particular, if every nonzero element of $\Lambda_\Delta$ has full support, then the $h^*$-polynomial coefficients satisfy the symmetry condition
\begin{equation}\label{eq:shifted}
    h_i^*(\Delta) = h_{d+1-i}^*(\Delta) \qquad \text{for} \quad 1 \leq i \leq d.
\end{equation}
\end{lemma}

\begin{proof}
Let $\lambda \in \Lambda_\Delta$. The inverse element $-\lambda$ in the group $\Lambda_\Delta$ is uniquely obtained by taking the additive inverse in $\mathbb{R}^{d+1}$ and reducing modulo $1$ coordinatewise. Therefore, the $j$-th coordinate of $-\lambda$ is $0$ if $\lambda_j = 0$, and $1 - \lambda_j$ if $\lambda_j > 0$ (which is precisely the condition that $j \in \operatorname{supp}(\lambda)$). Summing these relations over all coordinates $0 \leq j \leq d$ yields
\[\operatorname{ht}(\lambda) + \operatorname{ht}(-\lambda) = \sum_{j \in \operatorname{supp}(\lambda)} (\lambda_j + 1 - \lambda_j) + \sum_{j \notin \operatorname{supp}(\lambda)} (0 + 0) = |\operatorname{supp}(\lambda)|,
\]
which establishes \eqref{eq:inverse}.

For the second assertion, assume that every nonzero element of $\Lambda_\Delta$ has full support, meaning $|\operatorname{supp}(\lambda)| = d+1$ for all $\lambda \neq 0$. Under this hypothesis, \eqref{eq:inverse} simplifies to $\operatorname{ht}(\lambda) + \operatorname{ht}(-\lambda) = d+1$. The inversion map $\lambda \mapsto -\lambda$ is a group automorphism, hence a bijection on $\Lambda_\Delta$, and it restricts to a canonical bijection between the set of elements of height $i$ and the set of elements of height $d+1-i$. By \cref{lem:box}, the cardinality of the set of elements in $\Lambda_\Delta$ of height $k$ is given exactly by the coefficient $h_k^*(\Delta)$. Counting the elements in both height sets for $1 \leq i \leq d$ immediately yields \eqref{eq:shifted}.
\end{proof}

The following structural characterization reformulates the equivalence established in \cite[Theorem~2.1]{Hig10}.

\begin{proposition}{\em (Higashitani; \cite[Theorem~2.1]{Hig10})}\label{prop:equivalent}
Let $\Delta$ be a lattice simplex of dimension $d \geq 1$. Then the following conditions are equivalent:
\begin{enumerate}
    \item Every nonzero element $\lambda \in \Lambda_\Delta$ has full support.
    \item Every facet of $\Delta$ is unimodular with respect to its induced lattice.
    \item The $h^*$-polynomial coefficients of $\Delta$ satisfy the symmetry relations \eqref{eq:shifted}.
\end{enumerate}
\end{proposition}

By a simple derivation, we obtain the following lemma.

\begin{lemma}\label{lem:reflection}
Let $a_0,\ldots,a_r$ be nonnegative numbers, $r\geq1$.
Suppose $a_0\leq a_1$, $a_i=a_{r+1-i}$ for $1\leq i\leq r$, and $a_q\geq\cdots\geq a_r$, where $q=\floor{(r+1)/2}$.
Then $a_0\leq\cdots\leq a_q\geq\cdots\geq a_r$.
\end{lemma}

\begin{theorem}\label{thm:fullsupport}
Let $\Delta$ be a lattice simplex possessing the integer decomposition property (IDP). If every nonzero element of the coordinate group $\Lambda_\Delta$ has full support, then its $h^*$-vector is unimodal. Furthermore, this unimodality conclusion remains valid if the full-support condition is satisfied by the support face $F$ constructed in \cref{prop:core}.
\end{theorem}

\begin{proof}
If the group $\Lambda_\Delta$ is trivial, then $h^*_\Delta(t) = 1$, which is unimodal. Assume $\Lambda_\Delta$ is non-trivial. By \cref{lem:initial}, the first two coefficients satisfy the inequality $h_0^*(\Delta) \leq h_1^*(\Delta)$. The hypothesis that every nonzero element of $\Lambda_\Delta$ has full support allows us to use \cref{lem:inversion}, which establishes the shifted symmetry for the coefficients of $h^*_\Delta(t)$. By \cref{ext:tail} and \cref{lem:reflection}, the $h^*$-vector of $\Delta$ is unimodal.

For the second assertion, suppose the full-support condition holds for the support face $F$. By \cref{prop:core}, since $\Delta$ is IDP, the face $F$ inherits the IDP. We may therefore apply the first part of this theorem to $F$, concluding that the $h^*$-vector of $F$ is unimodal. The polynomial equality $h^*_F(t) = h^*_\Delta(t)$ established in \eqref{eq:core} implies that the $h^*$-vector of $\Delta$ is identical to the $h^*$-vector of $F$, up to the addition of trailing zeros. Therefore, the $h^*$-vector of $\Delta$ is unimodal.
\end{proof}

\begin{proof}[Proof of \cref{thm:main}]
Suppose that the normalized volume $p \coloneqq \operatorname{Vol}(\Delta)$ is a prime number. By \cref{prop:core}, we pass to the support face $F$, which satisfies $r \coloneqq \dim F \geq 1$. By the structural isomorphism in \eqref{eq:core}, the associated coordinate group $\Lambda_F$ has order $p$. Since $p$ is prime, $\Lambda_F$ is necessarily a cyclic group.

For each coordinate index $j$ of the face $F$, we define the natural projection homomorphism
\[\chi_j \colon \Lambda_F \longrightarrow \mathbb{R}/\mathbb{Z}, \quad \lambda \longmapsto \lambda_j \pmod{\mathbb{Z}}.
\]
\cref{prop:core} guarantees that no coordinate of $\Lambda_F$ is identically zero, which means that $\chi_j$ is a non-trivial homomorphism. The kernel $\ker(\chi_j)$ is a subgroup of $\Lambda_F$. Because the order of $\Lambda_F$ is prime, This dictates that $\ker(\chi_j)$ must be either the trivial subgroup or the entire group. As $\chi_j$ is non-trivial, it follows that $\ker(\chi_j) = \{0\}$.

Consequently, for every nonzero element $\lambda \in \Lambda_F$ and every index $j$, we have $\lambda_j \neq 0$. This demonstrates that all nonzero elements of $\Lambda_F$ have full support. Applying \cref{thm:fullsupport} concludes the proof, establishing the unimodality of the $h^*$-vector of $F$ and thereby that of the original simplex $\Delta$, correctly accounting for any terminal zeros.
\end{proof}

We note that the preceding argument applies seamlessly when $p=2$. 
In this scenario, the coordinate group $\Lambda_F$ has order two and thus contains a unique nonzero element, say $\lambda$. 
Since the face $F$ possesses the integer decomposition property, \cref{lem:nocarry} dictates that this element must have height $\operatorname{ht}(\lambda) = 1$. Furthermore, as an element of order two, $\lambda$ is its own additive inverse in $\Lambda_F$. Substituting $\lambda = -\lambda$ into \eqref{eq:inverse} yields an explicit support size:
\[|\operatorname{supp}(\lambda)| = \operatorname{ht}(\lambda) + \operatorname{ht}(-\lambda) = 1 + 1 = 2.
\]
Because $\lambda$ has full support on the $r$-dimensional face $F$, we simultaneously have $|\operatorname{supp}(\lambda)| = r+1$, which immediately forces $r=1$. Consequently, the support face $F$ is a one-dimensional lattice line segment of normalized volume two, and its associated $h^*$-polynomial is explicitly given by $h^*_F(t) = 1+t$.

\section{Extensions beyond prime volume}\label{Section--5}

In this section, we establish several sufficient conditions for the unimodality of the $h^*$-polynomial of IDP lattice simplices.

\begin{corollary}\label{cor:cyclic}
Let $\Delta$ be a lattice simplex possessing the integer decomposition property (IDP), and suppose that its associated coordinate group $\Lambda_\Delta$ is a non-trivial cyclic group of order $N$. Assume that a generator of $\Lambda_\Delta$ has coordinates of the form $a_j/N$, where $0 \leq a_j < N$ for each $j$, and that $\gcd(a_j, N) = 1$ whenever $a_j \neq 0$. Then the $h^*$-polynomial $h^*_\Delta(t)$ is unimodal.
\end{corollary}

\begin{proof}
By \cref{prop:core}, we effectively eliminate the identically zero coordinates by support-face reduction. On this support face $F$, the generator restricts to coordinates $a_j/N$ where $a_j \neq 0$ for all retained indices $j$. By hypothesis, this guarantees that $\gcd(a_j, N) = 1$ for every coordinate of the generator on $F$.

For any integer $1 \leq k < N$, consider the element obtained by taking the $k$-th multiple of the generator in the group $\Lambda_F$. The $j$-th coordinate of this element reduces to zero modulo $1$ if and only if $k (a_j/N) \in \mathbb{Z}$, which is equivalent to $N \mid k a_j$. Because $a_j$ and $N$ are coprime, this condition forces $N \mid k$. However, this contradicts the bounds $1 \leq k < N$.

Consequently, for any $1 \leq k < N$, no coordinate of the $k$-th multiple can vanish. This proves that every nonzero element of the coordinate group $\Lambda_F$ has full support. The unimodality of the $h^*$-vector then follows directly from \cref{thm:fullsupport}.
\end{proof}

We define a positive-dimensional lattice simplex to be \emph{clean} if its only boundary lattice points are its vertices.

\begin{corollary}\label{cor:clean}
Let $\Delta$ be a positive-dimensional lattice simplex possessing the integer decomposition property (IDP). Then $\Delta$ is clean if and only if every facet of $\Delta$ is unimodular. Furthermore, if $\Delta$ is clean, its $h^*$-vector is unimodal.
\end{corollary}

\begin{proof}
Suppose first that $\Delta$ is clean. By definition, the relative boundary of $\Delta$ contains no lattice points other than its vertices. Consequently, every facet of $\Delta$ contains exactly its vertices as lattice points. By \cref{lem:faces}, every facet naturally inherits the IDP from $\Delta$. Applying \cref{lem:initial} to these facets, we conclude that they must be unimodular.

Conversely, suppose that every facet of $\Delta$ is unimodular. For any unimodular simplex, the associated coordinate group is trivial, and formula \eqref{eq:basic} establishes that its total number of lattice points is exactly one greater than its dimension, which precisely equals its number of vertices. Thus, a unimodular facet contains no lattice points other than its vertices. Because the boundary of $\Delta$ is the union of its facets, it follows that every boundary lattice point of $\Delta$ is a vertex, proving that $\Delta$ is clean.

Finally, if $\Delta$ is a clean IDP simplex, the preceding argument guarantees that every facet of $\Delta$ is unimodular. By the equivalence established in \cref{prop:equivalent}, this ensures that every nonzero element of the coordinate group $\Lambda_\Delta$ has full support. By \cref{thm:fullsupport}, we conclude that the $h^*$-vector of $\Delta$ is unimodal.
\end{proof}

\begin{corollary}\label{cor:endpoint}
Let $P$ be a $d$-dimensional lattice polytope, with $d \geq 1$, possessing the integer decomposition property (IDP). If the coefficients of its $h^*$-polynomial satisfy $h_1^*(P) = h_d^*(P)$, then $P$ is a clean simplex, and its $h^*$-vector is unimodal.
\end{corollary}

\begin{proof}
By Equations \eqref{eq:basic} and \eqref{eq:last}, the number of lattice points on the relative boundary of $P$ can be explicitly computed as
\[|\partial P \cap M| = d + 1 + h_1^*(P) - h_d^*(P).
\]
Under the hypothesis that $h_1^*(P) = h_d^*(P)$, this enumeration simplifies precisely to $|\partial P \cap M| = d + 1$. 

The equality $|\partial P \cap M| = d + 1$ thus forces $P$ to have exactly $d + 1$ vertices, which constitute the entirety of the lattice points on its relative boundary. Having exactly $d + 1$ affinely independent vertices implies that $P$ is a simplex. Since its relative boundary contains no lattice points other than these vertices, $P$ is, by definition, a clean simplex. 

Because $P$ is a clean simplex that inherits the IDP, \cref{cor:clean} directly guarantees that its $h^*$-vector is unimodal.
\end{proof}






\noindent
{\small \textbf{Acknowledgments:}}
The authors would like to express sincere gratitude for all the suggestions that have improved the presentation of this paper.
Feihu Liu was partially supported by the Postdoctoral Fellowship Program and China Postdoctoral Science Foundation (Grant No. BX2026002).

\noindent{\small \textbf{Declaration of AI Assistance:}}

During the preparation of this manuscript, the authors utilized ChatGPT for language polishing and grammar checking to improve the readability and clarity of the text. All core reasoning and key conclusions were independently completed by the authors. The use of ChatGPT did not involve the substantive generation or creative contribution to the research content. 
In the final stage, we also used ChatGPT to verify the correctness of the manuscript.
The authors takes full academic responsibility for the entire manuscript.



\end{document}